\documentclass[11pt]{amsart}
\usepackage{amsmath}
\usepackage{CJK}
\allowdisplaybreaks
\usepackage{amssymb,latexsym}
\usepackage{url}
\usepackage{epsfig}
\usepackage[matrix,arrow]{xy}
\usepackage[usenames]{color}
\usepackage{slashed}
\usepackage{parskip}
\usepackage{amsmath,amssymb,graphics}
\usepackage{mathrsfs}
\usepackage{graphicx}

\newcommand{\mathsym}[1]{{}}

\newtheorem{thm}{Theorem}[section]
\newtheorem{cor}[thm]{Corollary}
\newtheorem{lem}[thm]{Lemma}
\newtheorem{prop}[thm]{Proposition}

\theoremstyle{definition}

\numberwithin{equation}{section}

\theoremstyle{example}

\newcommand{\be}{\begin{equation}}
\newcommand{\ee}{\end{equation}}
\newcommand{\bag}{\begin{eqnarray}}
\newcommand{\eag}{\end{eqnarray}}
\newcommand{\ban}{\begin{eqnarray*}}
\newcommand{\ean}{\end{eqnarray*}}
\newcommand{\ba}{\begin{aligned}}
\newcommand{\ea}{\end{aligned}}

\newcommand{\bpf}{\begin{proof} }
\newcommand{\epf}{\end{proof} }

\begin{document}
\title{The Continuity Equation of the Gauduchon Metrics}
\author{Tao Zheng}
\address{School of Mathematics and Statistics, Beijing Institute of Technology, Beijing 100081, China}
\email{zhengtao08@amss.ac.cn}
\begin{abstract}
We study the continuity equation of the Gauduchon metrics and establish its interval of maximal existence, which  extends the continuity equation of the K\"ahler metrics introduced by La Nave \& Tian for and of the Hermitian metrics introduced by Sherman \& Weinkove.
Our method is based on the solution to the Gauduchon conjecture by Sz\'ekelyhidi, Tosatti \& Weinkove.
\end{abstract}
\maketitle

\section{Introduction}\label{secintro}
Let $(M,J,g_0)$ be a closed (i.e., compact without boundary) Hermitian manifold with $\dim_{\mathbb{C}}M=n$  and the complex structure $J$, where $g_0$ is the   Hermitian metric, i.e., a Riemannian metric with $g_0(JX,JY)=g_0(X,Y)$ for all vector fields $X,\,Y\in \mathfrak{X}(M)$ (set of all the global and smooth vector fields). Then we can define a real $(1,1)$ form $\omega_0$  by
\begin{equation*}
\omega_0(X, Y):=g_0(JX, Y),\quad \forall\;X,\,Y\in \mathfrak{X}(M).
\end{equation*}
This form $\omega$ is determined uniquely by $g$ and vice versa. In what follows, we will not distinguish these two terms and we will not emphasize the complex structure $J$.

The Hermitian metric $\omega_0$ is called
\emph{K\"{a}hler} if
$
\mathrm{d}\omega_0=0,
$
\emph{Astheno-K\"{a}hler} (see \cite{jostyau,jostyauerratum}) if
$
\partial\bar\partial\omega_0^{n-2}=0,
$
\emph{balanced} (see \cite{michelsohn}) if
$
\mathrm{d}\omega_0^{n-1}=0,
$
\emph{Gauduchon} (see \cite{gauduchon1}) if
$
\partial\bar\partial\omega_0^{n-1}=0,
$
and \emph{strongly Gauduchon} (see \cite{popovici}) if
$
\bar\partial\omega_0^{n-1}\;\text{is}\;\partial\textrm-\text{exact}.
$

La Nave \& Tian \cite{lanavetianmathann} (cf. \cite{rubinsteinaim2008}) investigate a family of K\"ahler metrics
$\omega := \omega(s)$ satisfying the  continuity equation
\begin{equation}
\label{eqn: ce}
  \omega = \omega_0 - s \mathrm{Ric} (\omega), \quad \textrm{for}\quad s \geq 0,
\end{equation}
where
 $\mathrm{Ric}(\omega) = -\sqrt{-1}\partial\bar\partial \log \det (g_{i\bar j})$
is the Ricci form of the K\"ahler metric $\omega = \sqrt{-1} g_{i\bar{j}} \mathrm{d}z^i \wedge \mathrm{d}\bar{z}^j$ which is a real $(1,1)$ form.
This continuity equation can be viewed as an alternative to the K\"ahler-Ricci flow in carrying out the Song-Tian analytic minimal model program \cite{songtianinventiones2007,songtianinventiones2017}.
The Ricci curvature along the path is automatically bounded from below. This fact leads to several developments
\cite{fuguosongcrelle2019,lanavetianzhangajm2017,
liyanjga2018,zhangzhangimrn2019,zhangyashanmathann,zhangzhangimrn2020}.

Sherman \& Weinkove \cite{shermannweinkovejga2020} extend the continuity equation in \cite{lanavetianmathann} to Hermitian metrics and establish its interval of maximal existence and they also illustrate the behavior of this equation in the case of elliptic bundles over a curve of genus at least two.   This equation is closely related to the Chern-Ricci flow first introduced by \cite{gillcag} and studied deeply by Tosatti and Weinkove (and Yang) \cite{twjdg,twsurface,twymathann} (see also \cite{ftwzjfa,gillmmp,gillscalar,laurent,lr,niecrf2017,to2,yangcrf2016,zhengcjm}).

Li \& Zheng \cite{lizhengalmost} study the continuity equation of a family of almost Hermitian metrics and establish its interval of maximal existence. This continuity equation is closely related to the almost Chern-Ricci flow introduced by Chu, Tosatti \& Weinkove \cite{ctwjems} and furthermore studied by \cite{chu1607,zhengjga}.

In this paper we  investigate a natural analogue of (\ref{eqn: ce}) for the Gauduchon metrics, i.e., the continuity equation given by
\begin{equation}
\label{continuitygau}
\omega^{n-1}:=\omega_0^{n-1} - s (n-1)\left(\mathrm{Ric}(\omega)\wedge\alpha^{n-2}
-\Re\left(\sqrt{-1}\partial\log\frac{\omega^n}{\alpha^n}
\wedge\bar\partial(\alpha^{n-2})\right)\right)>0,\;\mbox{for}\;s\geq0,
\end{equation}
where $\alpha$ is the Gauduchon metric and $\mathrm{Ric}(\omega)$ is the  Chern-Ricci form of $\omega$ given by \eqref{chernricci} which coincides exactly with the Ricci form of $\omega$ when  $\mathrm{d}\omega=0$. If $\omega_0$ is the Gauduchon metric, then so is $\omega$ given by \eqref{continuitygau}.

When $n=2$, the continuity equation \eqref{continuitygau} is the same as the one in Sherman \& Weinkove \cite{shermannweinkovejga2020} in the Hermitian case and the one in La Nave \& Tian \cite{lanavetianmathann} in the K\"ahler case.
\begin{thm}
\label{mainthm}
Let $(M,\omega_0)$ be a closed Hermitian manifold with $\dim_{\mathbb{C}}M=n$ and $\omega_0$ a Hermitian metric.
Then there exists a unique family of Hermitian metrics
	$\omega = \omega(s)$ satisfying \eqref{continuitygau} for each $s\in[0,T)$,
where $T$ is defined by
\begin{align*}
T:=&\sup\Big\{s\geq0:\;\exists\; \varphi\in C^{\infty}(M,\,\mathbb{R}) \text{ such that}\\
&\quad\quad\quad\quad\quad\quad\Phi_s+\sqrt{-1}\partial\bar\partial\varphi\wedge\alpha^{n-2}
+\mathrm{Re}\left[\sqrt{-1}\partial\varphi\wedge\bar\partial(\alpha^{n-2})\right]>0\Big\},\nonumber
\end{align*}
with
\begin{equation*}
\Phi_s:=\omega_0^{n-1}-s(n-1)\left(\mathrm{Ric}(\omega_0)\wedge\alpha^{n-2}
-\Re\left(\sqrt{-1}\partial\log\frac{\omega_0^n}{\alpha^n}\wedge\bar\partial(\alpha^{n-2})\right)\right).
\end{equation*}
\end{thm}
When $n=2$, Theorem \ref{mainthm} is proved by Sherman \& Weinkove \cite{shermannweinkovejga2020} in the Hermitian case and by La Nave \& Tian \cite{lanavetianmathann} in the K\"ahler case.

The outline of the paper is as follows.  In Section \ref{section:pre} we establish preliminaries about Hermitian geometry for later use.  In Section \ref{section:pf}  we prove Theorems \ref{mainthm}.

\noindent {\bf Acknowledgements}
The author thanks Professor Jean-Pierre Demailly, Valentino Tosatti and Ben Weinkove for their invaluable directions and Yashan Zhang for very helpful conversations.

\section{Preliminaries}
\label{section:pre}
In this section, we collect some basic materials about   Hermitian geometry (see for example \cite{twjdg,zhengimrn}). Let $(M,J,g)$ be a Hermitian manifold with $\dim_{\mathbb{C}}M=n$, where $J$ is a complex structure and $g$ is the Hermitian metric respect to $J$, i.e., a Riemannian metric with $g(JX,JY)=g(X,Y)$ for any vector fields $X,\,Y\in \mathfrak{X}(M)$.
Then can define a real $(1,1)$ form $\omega$  by
$$
\omega(X, Y):=g(JX, Y),\quad \forall\;X,\,Y\in \mathfrak{X}(M).
$$
This form is determined uniquely by $g$ and vice versa. The Chern connection of $g$, denoted by $\nabla$, is the unique connection determined by $\nabla g=\nabla J=0$. The torsion of $\nabla$ is defined by
\begin{equation*}
T(X,Y):=\nabla_XY-\nabla_YX-[X,Y].
\end{equation*}
The curvature of $\nabla$ is defined by
\begin{equation}
R(X,Y)Z:=\nabla_X\nabla_YZ-\nabla_Y\nabla_XZ-\nabla_{[X,Y]}Z,
\end{equation}
with
$$
R(X,Y,Z,W):=g(R(X,Y)Z,W),\quad \forall\;W,X,Y,Z\in\mathfrak{X}(M).
$$
The Chern-Ricci curvature $\mathrm{Ric}$ is defined by
$$
\mathrm{Ric}(X,Y):=\mbox{trace of the map}\;Z\mapsto R(X,Y)Z.
$$
Then in the real local coordinate $x=(x^1,\cdots,x^{2n})$ with
$$J\left(\partial/\partial x^i\right)=\partial/\partial x^{n+i},\;J\left(\partial/\partial x^{n+i}\right)=-\partial/\partial x^i,\;i=1,\cdots, n,$$
we have
$$
g_{ij}=g_{n+i,n+j},\quad g_{i,n+j}=-g_{n+i,j},\quad g_{\alpha\beta}=g_{\beta\alpha},\quad i,\,j=1,\cdots,n,\;\alpha,\,\beta=1,\cdots,2n,
$$
where $g_{\alpha\beta}=g\left(\partial/\partial x^{\alpha}, \partial/\partial x^{\beta}\right)$.
Hence the complex local coordinate is given by $$z=(z^1,\cdots,z^n)=(x^1+\sqrt{-1}x^{n+1},\cdots,x^n+\sqrt{-1}x^{2n}).$$
We use the notation $\partial_i=\partial/\partial z^i,\;\partial_{\overline{j}}=\partial/\partial\overline{z}^j,\;i,\,j=1,\cdots,n$.   One can infer that
\begin{align}
g=&\sum\limits_{i,j=1}^n g_{i\overline{j}}\left(\mathrm{d}z^i\otimes \mathrm{d}\overline{z}^j+ \mathrm{d}\overline{z}^j\otimes\mathrm{d}z^i\right),\nonumber\\
\label{kahlerform}\omega=&\sqrt{-1}\sum\limits_{i,j=1}^n g_{i\overline{j}}\mathrm{d}z^i\wedge \mathrm{d}\overline{z}^j,
\end{align}
where $g_{i\overline{j}}=\frac{1}{2} \left(g_{i,j}+\sqrt{-1}g_{i,n+j}\right)$.

For each $(p,q)$ form
$$
\phi=\frac{1}{p!q!}\phi_{i_1\cdots i_p\overline{j_1}\cdots\overline{j_q}}\mathrm{d} z^{i_1}\wedge\cdots\wedge\mathrm{d} z^{i_p}\wedge\mathrm{d} \overline{z}^{j_1}\wedge\cdots\wedge\mathrm{d} \overline{z}^{j_q},
$$
the Hodge star operator $\ast$ with respect to the volume form $\frac{\omega^n}{n!}$ is given by  (see for example \cite{luqikeng})
\begin{align}\label{starformulacomplex}
\ast \phi =&\frac{(\sqrt{-1})^n(-1)^{np+\frac{n(n-1)}{2}}\det g}{(n-p)!(n-q)!p!q!}\phi_{i_1\cdots i_p\overline{j_1}\cdots\overline{j_q}}g^{\overline{\ell_1}i_1}
\cdots g^{\overline{\ell_p}i_p}g^{\overline{j_1}k_1}\cdots g^{\overline{j_q}k_q}\\
&\delta_{\ell_1\cdots \ell_pb_1\cdots b_{n-p}}^{1\cdots\cdots n}\delta_{k_1\cdots k_qa_1\cdots a_{n-q}}^{1\cdots\cdots n}
\mathrm{d} z^{a_1}\wedge\cdots\wedge\mathrm{d} z^{a_{n-q}}\wedge\mathrm{d} \overline{z}^{b_1}\wedge\cdots\wedge\mathrm{d}\overline{z}^{b_{n-p}}.\nonumber
\end{align}
A direct calculation yields that
$$
\ast 1=\frac{\omega^n}{n!},
\quad \overline{\ast \phi}= \ast\overline{ \phi},
$$
where the second equality shows that $\ast$ is a real operator.

It follows from \eqref{starformulacomplex} that
\begin{equation}
\label{astastpqform}
  \ast\ast \phi=(-1)^{p+q}\phi.
\end{equation}
Let us recall the concepts of positivity in for example \cite[Chapter III]{demaillybook1}.

A $(p,p)$ form $\varphi$ is said to be positive   if for any $(1,0)$ forms $\gamma_j,\,1\leq j\leq n-p$, then
$$
\varphi\wedge\sqrt{-1}\gamma_1\wedge\overline{\gamma_1}\wedge\cdots\wedge\sqrt{-1} \gamma_{n-p}\wedge\overline{\gamma_{n-p}}
$$
is a positive $(n,n)$ form. Any positive $(p,p)$ form $\varphi$ is real, i.e., $\overline{\varphi}=\varphi$. In particular,   a real $(1,1)$ form given by
\begin{align}\label{11}
\phi=\sqrt{-1}\phi_{i\overline{j}}\mathrm{d} z^i\wedge\mathrm{d}\overline{z}^j
\end{align}
is positive if and only if $(\phi_{i\overline{j}})$ is a semi-positive Hermitian matrix and we denote $\det \phi:=\det(\phi_{i\overline{j}})$.

A real $(n-1,n-1)$ form given by
\begin{align}\label{n-1}
\psi=&(\sqrt{-1})^{n-1}\sum\limits_{i,j=1}^n(-1)^{\frac{n(n+1)}{2}+i+j+1}\psi^{\overline{j}i}\\
&\mathrm{d} z^1\wedge\cdots\wedge\widehat{\mathrm{d} z^i}\wedge\cdots\wedge\mathrm{d} z^n\wedge \mathrm{d} \overline{z}^1\wedge\cdots\wedge\widehat{\mathrm{d}\overline{z}^j}\wedge\cdots\wedge\mathrm{d} \overline{z}^n\nonumber
\end{align}
is positive if and only if $(\psi^{\overline{j}i})$ is a semi-positive Hermitian matrix and we denote $\det\psi:=\det(\psi^{\overline{j}i})$.
We remark that one can call a real $(1,1)$ form $\phi$ ( resp. a real $(n-1,n-1)$ form $\psi$) strictly positive
if the
Hermitian matrix $(\phi_{i \overline{j}})$ (resp. $(\psi^{\overline{j} i})$) is positive
definite.

For a strictly positive $(1,1)$ form $\phi$ defined as in \eqref{11}, we can deduce a strictly positive $(n-1,n-1)$ form
\begin{align}\label{11n1n1formula}
\frac{\phi^{n-1}}{(n-1)!}=&(\sqrt{-1})^{n-1}\sum\limits_{k,\ell=1}^n(-1)^{\frac{n(n+1)}{2}+k+\ell+1}\mathrm{det}(\phi_{i\overline{j}})\tilde{\phi}^{\overline{\ell}k}\\
&\mathrm{d} z^{ 1}\wedge\cdots\wedge\widehat{\mathrm{d} z^k}\wedge \cdots\wedge\mathrm{d} z^{n}\wedge\mathrm{d}\overline{z}^{ 1}\wedge\cdots\wedge\widehat{\mathrm{d} \overline{z}^{\ell}}\wedge\cdots\wedge\cdots\wedge \mathrm{d}\overline{z}^{n}\nonumber
\end{align}
where $(\tilde{\phi}^{\overline{\ell}k} )$ is the inverse matrix of $(\phi_{i\overline{j}})$, i.e.,  $\sum\limits_{\ell=1}^n\tilde{\phi}^{\overline{\ell}j}\phi_{k\overline{\ell}}=\delta_{k}^j$. Hence we have
\begin{align}\label{detchin-1}
\det\left(\frac{\phi^{n-1}}{(n-1)!}\right)=\left(\det\phi\right)^{n-1}.
\end{align}
Hence, if $\xi$ is another real $(1,1)$ form with $\det\xi\neq0$, then we can deduce
\begin{equation}\label{astdet}
\frac{\det(\ast\phi)}{\det(\ast\xi)}=\frac{\det\phi}{\det\xi}.
\end{equation}
The Christoffel symbols of $\nabla$ are denoted by
$$
\Gamma_{\alpha\beta}^\gamma:=\mathrm{d}z^\gamma(\nabla_\alpha\partial_\beta),\quad
\alpha,\beta,\gamma\in\{1,\cdots,n,\bar 1,\cdots,\bar n\},
$$
where we use the notation that $\mathrm{d}z^{\bar i}=\mathrm{d}\bar z^i$ with $1\leq i\leq n$. The non-zero components of the Christoffel symbols of $\nabla$ are
$$
\Gamma_{ij}^k=g^{k\bar \ell}\partial_i g_{j\bar \ell},\quad \Gamma_{\bar i\bar j}^{\bar k}=\overline{\Gamma_{ij}^k},\quad 1\leq i,j,k\leq n.
$$
We also use the notations
$$
T^\alpha_{\beta\gamma}:=\mathrm{d}z^\alpha(T(\partial_\beta,\partial_\gamma)),\quad
R_{\alpha\beta\gamma}{}^\delta:=\mathrm{d}z^\delta(R(\partial_\alpha,\partial_\beta)\partial_\gamma),\quad
\alpha,\beta,\gamma,\delta\in\{1,\cdots,n,\bar 1,\cdots,\bar n\}.
$$
Then one infers
$$
T_{ij}^k=\Gamma_{ij}^k-\Gamma_{ji}^k,\quad R_{i\bar jk}{}^p=-\partial_{\bar j}\Gamma_{ik}^p,\quad
R_{i\bar jk\bar\ell}:=R_{i\bar jk}{}^pg_{p\bar \ell}.
$$
The Chern-Ricci form is defined by
\begin{equation}
\label{chernricci}
\mathrm{Ric}(\omega):=\sqrt{-1}R_{i\bar jp}{}^p\mathrm{d}z^i\wedge\mathrm{d}\bar z^j
=-\sqrt{-1}\partial\bar\partial\log\det(g_{i\bar j}).
\end{equation}

\section{Proof of the Main Theorem}
\label{section:pf}
 In order to prove Theorem \ref{mainthm}  we need reduce
the equation  \eqref{continuitygau}
to a  complex Monge-Amp\`ere type equation on $M$.
For each $\hat T \in (0,T)$, the definition of $T$ yields that there is
a smooth function $\varphi$ such that
$$\Phi_{\hat T}+\sqrt{-1}\partial\bar\partial\varphi\wedge\alpha^{n-2}
+\Re\left[\sqrt{-1}\partial\varphi\wedge\bar\partial(\alpha^{n-2})\right]>0,$$
with
\begin{equation*}
\Phi_{\hat T}:=\omega_0^{n-1}-\hat{T}(n-1)\left(\mathrm{Ric}(\omega_0)\wedge\alpha^{n-2}
-\Re\left(\sqrt{-1}\partial\log\frac{\omega_0^n}{\alpha^n}\wedge\bar\partial(\alpha^{n-2})\right)\right).
\end{equation*}
Let $\Omega$ be the volume form given by
\begin{equation}
\label{defnvolOmega}
\Omega =
\omega_0^n e^{\frac{\varphi}{(n-1)\hat T}}.
\end{equation}
Then it follows from \eqref{chernricci}  that
$$\omega_0^{n-1}-\hat{T}(n-1)\left(\mathrm{Ric}(\Omega)\wedge\alpha^{n-2}
-\Re\left(\sqrt{-1}\partial\log\frac{\Omega}{\alpha^n}\wedge\bar\partial(\alpha^{n-2})\right)\right)>0.$$
The convexity of the space of Hermitian matrices (cf. \eqref{n-1}) yields that
$$
\hat\omega_s^{n-1}:=\omega_0^{n-1}-s(n-1)\left(\mathrm{Ric}(\Omega)\wedge\alpha^{n-2}
-\Re\left(\sqrt{-1}\partial\log\frac{\Omega}{\alpha^n}\wedge\bar\partial(\alpha^{n-2})\right)\right)>0,\quad \forall\;s\in[0,\hat T].
$$
\begin{prop}
\label{prop}
Let $(M,\omega_0)$ be a closed Hermitian manifold with $\dim_{\mathbb{C}}M=n$ and $\omega_0$ a Hermitian metric. Then for $s\in [0,\hat T]$ fixed,  there exists a Gauduchon metric
$\omega$ satisfying  \eqref{continuitygau}
if and only if
there exists a smooth function $u\in C^\infty(M,\mathbb{R})$ satisfying
\begin{equation}
\label{ma}
\log\frac{\det\omega^{n-1}}{e^{\frac{\varphi}{\hat T}}\det\omega_0^{n-1}}- u =  0,
\end{equation}
with
\begin{equation}
\label{ma1}
 \omega^{n-1}:=\hat\omega_s^{n-1}+s\left(\sqrt{-1}\partial\bar\partial u\wedge\alpha^{n-2}+\Re\left(\sqrt{-1}\partial u\wedge\bar\partial(\alpha^{n-2})\right)\right)>0.
\end{equation}
\end{prop}
\begin{proof}
For the `if' direction, we suppose that $\omega:=\omega(s)$ satisfying \eqref{continuitygau}.  We define $u$ by
$$
u=\log\frac{\det\omega^{n-1}}{e^{\frac{\varphi}{\hat T}}\det\omega_0^{n-1}}.
$$
Then it follows from \eqref{chernricci} and \eqref{defnvolOmega} that
\begin{align*}
&(n-1)\mathrm{Ric}(\omega)-(n-1)\mathrm{Ric}(\Omega)\\
=&-(n-1)\sqrt{-1}\partial\bar\partial\log\frac{\omega^n}{\omega_0}+\frac{1}{\hat T}\sqrt{-1}\partial\bar\partial\varphi\\
=&-\sqrt{-1}\partial\bar\partial\log\frac{\det\omega^{n-1}}{e^{\frac{\varphi}{\hat T}}\det\omega_0^{n-1}}
=-\sqrt{-1} \partial\bar\partial u,
\end{align*}
which yields that
\begin{align*}
\omega^{n-1}
:=&\omega_0^{n-1}-s(n-1)\left(\mathrm{Ric}(\omega)\wedge\alpha^{n-2}
-\Re\left(\sqrt{-1}\partial\log\frac{\omega^n}{\alpha^n}
\wedge\bar\partial(\alpha^{n-2})\right)\right)\\
=&\omega_0^{n-1}-s(n-1)\left(\mathrm{Ric}(\Omega)\wedge\alpha^{n-2}
-\Re\left(\sqrt{-1}\partial\log\frac{\Omega}{\alpha^n}
\wedge\bar\partial(\alpha^{n-2})\right)\right)\\
&+s\sqrt{-1}\partial\bar\partial u\wedge\alpha^{n-2}
+s\Re\left(\sqrt{-1}\partial u\wedge\bar\partial(\alpha^{n-2})\right)\\
=&\hat\omega_s^{n-1}+s\sqrt{-1}\partial\bar\partial u\wedge\alpha^{n-2}
+s\Re\left(\sqrt{-1}\partial u\wedge\bar\partial(\alpha^{n-2})\right),
\end{align*}
as desired.

For the `only if' direction, if $u$ satisfies \eqref{ma}-\eqref{ma1}, then a direct calculation, together with \eqref{chernricci}, yields that
$\omega$ satisfies \eqref{continuitygau}.
\end{proof}
An immediate consequence of Proposition \ref{prop} is the uniqueness of solutions to the continuity equation \eqref{continuitygau}.
\begin{cor}
\label{corunique}
Let $(M,\omega_0)$ be a closed Hermitian manifold with $\dim_{\mathbb{C}}M=n$ and $\omega_0$ a Hermitian metric. Then if $\omega'$ and $\omega$ are two almost Hermitian metrics solving the continuity equation  \eqref{continuitygau}  for the same $s$ in $[0,T)$, then $\omega'=\omega$.
\end{cor}
\begin{proof} For $s=0$, there is nothing to prove.  For $s \in (0,T)$, it suffices to prove the uniqueness of solutions  to the equation  \eqref{ma}-\eqref{ma1} by Proposition \ref{prop}. We assume that both $u$ and $u'$ are the solutions to \eqref{ma}-\eqref{ma1}. We set $\theta:=u'-u$. Then it follows from \eqref{ma} that
\begin{equation*}
 \log \frac{ \det\left(\omega_u^{n-1}+s\sqrt{-1}\partial\bar\partial \theta\wedge\alpha^{n-2}
+s\Re\left(\sqrt{-1}\partial \theta\wedge\bar\partial(\alpha^{n-2})\right)\right)  }{ \det\omega_u^{n-1}} =\theta,
\end{equation*}
where
$$
\omega_u^{n-1}:=\hat\omega_s^{n-1}+s\sqrt{-1}\partial\bar\partial u\wedge\alpha^{n-2}
+s\Re\left(\sqrt{-1}\partial u\wedge\bar\partial(\alpha^{n-2})\right).
$$
Since at the point where $\theta$ attains its maximum (resp. minimum) one has
$$
\mathrm{d}\theta =0,\quad \sqrt{-1}\partial\bar\partial \theta\leq 0\quad (\mbox{resp.}\;\geq 0),
$$
we can deduce that $\theta\equiv0$, as desired.
\end{proof}
Note that  \eqref{ma}-\eqref{ma1}  is trivially solved when
$s=0$ by taking
$$
u=\log\frac{\det\omega_0^{n-1}}{e^{\frac{\varphi}{\hat T}}\det\omega_0^{n-1}}.
$$
Fix $s \in (0,\hat{T}]$.  We define a new function $\psi = su$
and a function $G= \log \frac{e^{\frac{\varphi}{\hat T}}\det\omega_0^{n-1}}{\det\alpha^{n-1}}$.  Then the equation  \eqref{ma}-\eqref{ma1}
becomes
\begin{equation}
\label{ma2}
\log\frac{\det\omega^{n-1}}{\det\alpha^{n-1}}=\frac{\psi}{s}+G,
\end{equation}
with
\begin{equation}
\label{ma3}
 \omega^{n-1}:=\hat\omega_s^{n-1}+ \left(\sqrt{-1}\partial\bar\partial \psi\wedge\alpha^{n-2}+\Re\left(\sqrt{-1}\partial \psi\wedge\bar\partial(\alpha^{n-2})\right)\right)>0.
\end{equation}
It follows from \eqref{astastpqform} and \eqref{astdet} that
\begin{align}
\label{ma4}
\log\frac{\det\omega^{n-1}}{\det\alpha^{n-1}}
=&\log\frac{\det\left(\frac{\omega^{n-1}}{(n-1)!}\right)}
{\det\left(\frac{\alpha^{n-1}}{(n-1)!}\right)}
= \log\frac{\det\left(\ast\ast\frac{\omega^{n-1}}{(n-1)!}\right)}
{\det\left(\ast\ast\frac{\alpha^{n-1}}{(n-1)!}\right)} \\
=&\log\frac{\det\left(\ast\frac{\omega^{n-1}}{(n-1)!}\right)}
{\det\left(\ast\frac{\alpha^{n-1}}{(n-1)!}\right)}
= \log\frac{\det\left(\ast\frac{\omega^{n-1}}{(n-1)!}\right)}
{\det\alpha}.\nonumber
\end{align}
Here and henceforth, $\ast$ is the Hodge star operator with respect to $\frac{\alpha^n}{n!}$.

Recall that $s$ here is fixed.   Then Theorem \ref{mainthm} follows from Propositin \ref{prop}, \eqref{ma}, \eqref{ma1}, \eqref{ma2}, \eqref{ma3}, \eqref{ma4} and the following result.
\begin{thm}
\label{thmuse}
Let $(M,\omega_h)$ be a closed almost Hermitian manifold with $\dim_{\mathbb{C}}M=n$ and $\omega_h$ a  Hermitian metric.  Then for $G\in C^\infty(M,\mathbb{R})$ and $\lambda>0$ a constant, there exists a unique solution $\varphi$ to the equation
\begin{equation}
\label{ma5}
\log \frac{\left(\varpi+\frac{1}{n-1}\left[(\Delta \varphi)\alpha-\sqrt{-1}\partial\bar\partial \varphi\right]+Z(\mathrm{d}\varphi)\right)^n}{\alpha^n} = \lambda \varphi +G,
\end{equation}
where
$
\varpi=\frac{1}{(n-1)!}\ast \omega_h^{n-1},
$
$
\Delta \varphi=\alpha^{\overline{j}i}\partial_i\partial_{\overline{j}}\varphi,
$
\begin{align}
\label{tildeomega}
\tilde{\omega}:=\varpi+\frac{1}{n-1}\left[(\Delta \varphi)\alpha-\sqrt{-1}\partial\bar\partial \varphi\right]+Z(\mathrm{d}\varphi)
=:\sqrt{-1} \tilde g_{i\overline{j}}\mathrm{d} z^i\wedge\mathrm{d} \overline{z}^j>0,
\end{align}
and
\begin{align}
\label{zu}
Z(\mathrm{d}\varphi)=\frac{1}{(n-1)!}\ast\Re\left(\sqrt{-1}\partial \varphi\wedge\bar\partial(\alpha^{n-2})\right).
\end{align}
\end{thm}
\begin{proof}
For $\lambda=0$, Equation \eqref{ma5} is solved by \cite{stw1503} (cf.\cite{zhengimrn}). Here we use the method modified from \cite{stw1503} (cf.\cite{fenggezhengcmabd1}).

We use the method of continuity to solve \eqref{ma5}. We study a family of equations for $ \varphi:=\varphi(t)  \in  C^{2,\gamma}(M,\mathbb{R})$ for some $\gamma\in(0,1)$ fixed
\begin{equation}
\label{mat}
\log \frac{\left(\varpi+\frac{1}{n-1}\left[(\Delta \varphi )\alpha-\sqrt{-1}\partial\bar\partial \varphi \right]+Z(\mathrm{d}\varphi )\right)^n}{\alpha^n} = \lambda \varphi +(1-t)G_0 +tG,
\end{equation}
where
$
G_0:=\log \frac{\varpi^n}{\alpha^n}
$.
We set
\begin{equation*}
\mathscr{T}:=\Big\{t\in [0,1]:\,\mbox{there exists}\, \varphi \in C^{2,\gamma}(M,\mathbb{R})\,\mbox{solves \eqref{mat}}  \Big\}.
\end{equation*}
The definition of $G_0$ yields that $0\in \mathscr{T}$ since we can take $\varphi(0)=0$.
It suffices to show that $\mathscr{T}$ is both open and closed.

For the openness of $\mathscr{T}$, we consider the map
\begin{align*}
\Psi:\,& [0,1] \times C^{2,\gamma}(M,\mathbb{R}) \rightarrow C^{\gamma}(M,\mathbb{R}),\\
&(t, \varphi )\mapsto\log \frac{\left(\varpi+\frac{1}{n-1}\left[(\Delta \varphi)\alpha-\sqrt{-1}\partial\bar\partial \varphi\right]+Z(\mathrm{d}\varphi)\right)^n}{\alpha^n} - \lambda \varphi-(1-t)G_0 -tG.
\end{align*}

Assume $t_0 \in \mathscr{T}$, and that (\ref{mat}) has a corresponding solution $\varphi_{0}$.   Write
$$
\tilde\omega_0:=\varpi+\frac{1}{n-1}\left[(\Delta \varphi_0)\alpha-\sqrt{-1}\partial\bar\partial \varphi_0\right]+Z(\mathrm{d}\varphi_0)
=:\sqrt{-1} h_{i\overline{j}}\mathrm{d} z^i\wedge\mathrm{d} \overline{z}^j>0.
$$
The derivative of  $\Psi$ in the second variable at $(t_0, \varphi_0)$ is the linear operator
$B: C^{2,\gamma}(M,\mathbb{R}) \rightarrow C^{\gamma}(M,\mathbb{R})$ given by
$$B (u) = P(u) - \lambda u,$$
where
\begin{equation}
\label{pu}
P(u):=\Theta^{\bar ji} \partial_{\bar j}\partial_iu+h^{\bar ji}\left(Z_{i\bar j}^p\partial_pu+\overline{Z_{j\bar i}^p}\partial_{\bar p}u\right).
\end{equation}
with
$$
\Theta^{\bar ji}=\frac{1}{n-1}\left((\mathrm{tr}_{\tilde\omega_0}\alpha)\alpha^{\bar ji}-h^{\bar ji}\right).
$$
Since $(\Theta^{\bar ji})>0$, both $P$ and $B$ are strictly elliptic operators.

We define an operator $B_0$ by
$$
B_0:\,C^{2,\gamma}(M,\mathbb{R}) \rightarrow C^{\gamma}(M,\mathbb{R}),\quad
B_0(u):=\Delta_{\Theta}u-\lambda u,
$$
where $\Delta_{\Theta}$ is the Laplace-Beltrami operator of the Riemannian metric $\Theta$.
Since $\lambda>0$, it follows from \cite[Theorem 4.18]{aubinma1982} that $B_0$ is an isomorphism map.

Let us consider a family of strictly elliptic operators $B_t:=(1-t)B_0+tB$ for $t\in[0,1]$. For each $u\in C^{2,\gamma}(M,\mathbb{R})$,  the Schauder theory (see for example \cite{gt1998}) yields that
\begin{equation}
\label{schauderbt}
\|u\|_{C^{2,\gamma}(M,\mathbb{R})}\leq C\left(\sup_M|B_t(u) |+\|B_t(u)\|_{C^\gamma(M,\mathbb{R})}\right)
\end{equation}
for some uniform constant $C>0$ independent of $t$.
Since $B_0$ is an isomorphism map, it follows from \eqref{schauderbt} and \cite[Thoerem 5.2]{gt1998} that $B_t$ is isomorphism for each $t\in [0,1]$ and so is $B$, which, together with the Inverse Function Theorem, yields that  $\mathscr{T}$ is open.


For the closeness of $\mathscr{T}$, we need a priori  estimates on $\varphi$ solving \eqref{mat} independent of $t$.

\textbf{A uniform bound }
\begin{equation}
\label{equzeroestimate}\sup_M|\varphi|  \leq C
\end{equation}
follows immediately from the maximum principle. Here and henceforth, $C$ will denote a uniform constant independent of $t$ that may change from line to line.

We denote $h:=\lambda\varphi+(1-t)G_0 +tG$.

\textbf{For the second order estimate}, we claim
\begin{equation}
\label{equ2ndestimate}
\sup_M|\partial\bar\partial \varphi|_\alpha\leq CK,
\end{equation}
with $K:=1+\sup_M|\partial \varphi|_\alpha^2$.

We will deduce the estimate \eqref{equ2ndestimate}  by making use of  \cite{stw1503} (cf.\cite{gaborjdg,fenggezhengcmabd1}). The main difference is that when we apply covariant derivative  to both sides of \eqref{ma5}, $\nabla_ih$ is bounded by $O(1+|\partial \varphi|_\alpha)$ rather that $O(1)$ and $\nabla_{\bar j}\nabla_i h$ is bounded by $O(\lambda_1)$ rather than $O(1)$, which are harmless in the following arguments.
For the sake of completeness, we include here a brief sketch of the proof from \cite{fenggezhengcmabd1,stw1503} motivated by \cite{twarxiv1906}.

We need some preliminaries. For each real $(1,1)$ form $\xi$, we define
$$
P_{\alpha}(\xi):=\frac{1}{n-1}\Big((\mathrm{tr}_{\alpha}\xi)\alpha-\xi\Big)=\frac{1}{(n-1)!}\ast(\xi\wedge\alpha^{n-2}).
$$
A direct calculation yields that $\mathrm{tr}_{\alpha}\xi=\mathrm{tr}_{\alpha}\left(P_{\alpha}(\xi)\right)$ and that
$$
\xi=\left(\mathrm{tr}_{\alpha}\left(P_{\alpha}(\xi)\right)\right)\alpha-(n-1)P_{\alpha}(\xi).
$$
We set
\begin{align*}
\chi_{i\overline{j}}=&(\mathrm{tr}_{\alpha}\varpi)\alpha_{i\overline{j}}-(n-1)\varpi_{i\overline{j}},
\end{align*}
with $P_{\alpha}(\chi)=\varpi$, and
$$
W_{i\overline{j}}(\mathrm{d}\varphi)=\left(\mathrm{tr}_{\alpha}Z(\mathrm{d}\varphi)\right)
\alpha_{i\overline{j}}-(n-1)Z_{i\overline{j}}(\mathrm{d}\varphi)
=:W_{i\overline{j}}^p\varphi_p+\overline{W^p_{j\overline{i}}}\varphi_{\overline{p}}.$$
Note that
\begin{equation}
 \label{eq:c1} Z(\mathrm{d}\varphi):= P_\omega(W(\mathrm{d}\varphi))
 = \frac{1}{n-1}\Big( (\mathrm{tr}_\omega
  W(\mathrm{d}\varphi))\omega - W(\mathrm{d}\varphi)\Big).
 \end{equation}
Then we set
\begin{align}\label{defnvartheta}
g_{i\overline{j}}
=\chi_{i\overline{j}}+\varphi_{i\overline{j}}+W_{i\overline{j}}(\mathrm{d}\varphi)
\end{align}
with $P_{\alpha}(g_{i\overline{j}})=\tilde g_{i\overline{j}}$.

In orthonormal coordinates for $\alpha$ at each given point,  it follows that the component $Z_{i\overline{j}}$ is independent of $\varphi_{\overline{i}}$ and $u_{j}$, and that $\nabla_{i} Z_{i\overline{i}}$ is independent of $\varphi_i,\varphi_{\overline{i}}, \varphi_{i\overline{i}},\nabla_{i}\varphi_i $ (see \cite{stw1503}).

Moreover, $\nabla_i\nabla_{\overline{i}} Z_{i \overline{i}}$ is independent of $\nabla_i\varphi_i,\nabla_{\overline{i}}\varphi_{\overline{i}},
\nabla_{\overline{i}}\nabla_i\varphi_i$ and $\nabla_{\overline{i}}\varphi_{i\overline{i}}$. For each index $p$, $\nabla_iZ_{p\overline{i}}$ is independent of $\nabla_{i}\varphi_i$. In what follows, we will use these properties directly and do not prove them again.

Denote $\left(\tilde B_{i}{}^j\right)=\left(\tilde g_{i\overline{\ell}}\alpha^{\overline{\ell}j}\right)$ which can be viewed as the endomorphism of $T^{1,0}M$. This endomorphism is Hermitian with respect to the Hermitian metric $\alpha$.
We define
$$
\tilde F(\tilde B)=\log \det \tilde B=\log(\mu_1\cdots\mu_n)=:\tilde f(\mu_1,\cdots,\mu_n),
$$
where $\mu_1,\cdots,\mu_n$ denote the eigenvalues of $\left(\tilde B_{i}{}^j\right)$. Then \eqref{ma5} can be rewritten as
\begin{align}\label{udef2nd1}
\tilde F (\tilde B)=h.
\end{align}
For   $\tilde f$ and $h$,  there holds
\begin{enumerate}
\item[(i)] $\tilde f$ is defined on
$$
\Gamma_n=\Big\{(x_1,\cdots,x_n)\in \mathbb{R}^n:\;x_i>0,\;1\leq i\leq n\Big\}.
$$
\item[(ii)] $\tilde f$ is symmetric, smooth, concave and increasing, i.e., $\tilde f_i>0$ for $1\leq i\leq n$.
\item[(iii)] $\sup_{\partial \Gamma_n}\tilde f\leq \inf_{M}h$.
\item[(iv)] For each $\mu \in \Gamma_n$, we get $\lim_{t\rightarrow\infty}\tilde f(t\mu)=\sup_{\Gamma_n}\tilde f=\infty$.
\item[(v)]$h$ is bounded on $M$ because of the uniform estimate  $\sup_M|\varphi|\leq C$.
\end{enumerate}
We also define
$$
F (A):=\tilde F(\tilde B)=:f(\lambda_1,\cdots,\lambda_n),
$$
where $(A_{i}{}^{j})=\left(g_{i\overline{\ell}}\alpha^{\overline{\ell}j}\right)$, which is also an endomorphism of $T^{1,0}M$   with respect to the Hermitian metric $\alpha$, and let $\lambda_1,\cdots,\lambda_n$ denote its eigenvalues. There exists a map
\begin{align}
\label{Pdefn}
P:\;\mathbb{R}^n\longrightarrow \mathbb{R}^n,\quad \mu_k=\frac{1}{n-1}\sum\limits_{i\neq k}\lambda_{i},
\end{align}
induced   by $P_{\alpha}$ above. Then we have
$$
f(\lambda_1,\cdots,\lambda_n)=\tilde f\circ P(\lambda_1,\cdots,\lambda_n)
$$
defined on $\Gamma:=P^{-1}(\Gamma_n)$. Note that  $f$  satisfies the same conditions as $\tilde f$. Then \eqref{udef2nd1} can also be rewritten as
\begin{equation}
\label{nablakh}
F(A)=h.
\end{equation}
Since
\begin{align}\label{daotildef}
\tilde f_i=\frac{1}{\mu_i},\quad 1\leq i\leq n,
\end{align}
we can get
\begin{align}\label{fk}
f_i=\frac{1}{n-1}\sum\limits_{k\neq i}\frac{1}{\mu_k},\quad 1\leq i\leq n.
\end{align}
A direct calculation shows that
\begin{equation}
\label{lambdamudefn}
\mu_j=\frac{1}{n-1}\sum_{k\not=j}\lambda_k,\quad \lambda_j=\sum_{k=1}^n\mu_k-(n-1)\mu_j.
\end{equation}
Suppose that $\lambda_1\geq \lambda_2\geq\cdots\geq \lambda_n\in \Gamma$. From the definition of $P$, \eqref{daotildef} and \eqref{fk}, we have
\begin{align}
0 <\mu_1\leq \mu_2\leq \cdots\leq \mu_n,\nonumber\\
\tilde f_1\geq \tilde f_2\geq \cdots\geq \tilde f_n,\nonumber\\
0<f_1\leq f_2\leq\cdots\leq f_n,\nonumber\\
\label{daoshueigenn}\sum\limits_{k=1}^n\lambda_kf_k=\sum\limits_{k=1}^n\mu_k\tilde f_k=n.
\end{align}
where we also use the fact that $\left(\tilde B_i{}^j\right)$ is positive definite. For $k\geq 2$, we have
\begin{align}\label{eq:fk1}
  0<\frac{\tilde f_1}{n-1}\leq f_k\leq \tilde f_1.
\end{align}
and
\begin{align}\label{eq:fk2}
\tilde f_k\leq (n-1)f_1,\quad k>1.
\end{align}
We have (see \cite{gaborjdg,stw1503})
\begin{prop}
For each $\mathbf{x}\in M$, choose orthonormal coordinates for $\alpha$ at $\mathbf{x}$, with $g$ defined as in \eqref{defnvartheta} is diagonal with eigenvalues $(\lambda_1,\cdots,\lambda_n)$. Then there exist uniform constants $R>0$ and $\kappa\in(0,1)$ such that if
$$
|\lambda|\geq R,
$$
then for $f(\lambda)=h$, there  holds two possibilities as follows.
\begin{enumerate}
\item[(a)]  We have
\begin{align}\label{partaformula}
\sum\limits_kf_{k}(\lambda)(\chi_{k\overline{k}}-\lambda_k)>\kappa \sum\limits_kf_k(\lambda).
\end{align}
\item[(b)]  Or we have
\begin{align}
f_k(\lambda)>\kappa\sum\limits_{i=1}^nf_{i}(\lambda)
\end{align}
for all $k=1,2,\cdots,n$.
\end{enumerate}
In addition, we have
\begin{equation}
\label{ftau}
\mathcal{F}(\lambda):=\sum_{i=1}^nf_i(\lambda)\geq \kappa>0.
\end{equation}
\end{prop}
We also need
\begin{lem}[Gerhardt \cite{gerhardt}]
\label{lemgerhardt}
If $F(A)=f(\lambda_1,\cdots, \lambda_n)$ in terms of a symmetric funtion of the eigenvalues, then at a diagonal matrix $(A_{i}{}^j)$ with distinct eigenvalues we have
\begin{align}
\label{f1stdaoshu}F^{ij}=&\delta_{ij} f_i,\\
\label{f2nddaoshu}F^{ij,rs}=&f_{ir}\delta_{ij}\delta_{rs}+\frac{f_i-f_j}{\lambda_i-\lambda_j}(1-\delta_{ij})\delta_{is}\delta_{jr},
\end{align}
where
$$
F^{ij}=\frac{\partial F}{\partial  A_{i}{}^j}, \quad F^{ij,rs}=\frac{\partial^2F}{\partial  A_{i}{}^j\partial  A_{r}{}^s}.
$$
\end{lem}
We write $F^{i\bar q}:=F^{ij }\alpha^{\bar q j}$.

It suffices to prove
$$\lambda_1\leq CK. $$
Indeed, since $\sum\limits_{i=1}^n\lambda_i=\sum\limits_{i=1}^n\mu_i>0$, if $\lambda_1\leq CK$ then so is $|\lambda_k|,\,k=2,\cdots,n$, which yields \eqref{equ2ndestimate}. We consider the quantity
$$
H(\mathbf{x}):= \log \lambda_1(\mathbf{x})+\varsigma(|\partial \varphi|_g^2(\mathbf{x})) +\psi(\varphi(\mathbf{x})),\quad \forall\,\mathbf{x}\in M,
$$
where we set
 $$
\varsigma(s)=-\frac{1}{2}\log\left(1-\frac{s}{2K}\right),\quad\psi(s)=D_1e^{-D_2s},
$$
with sufficiently large uniform constants $D_1,D_2>0$ to be determined later.
A direct calculation shows that
$$
\varsigma\left(|\partial u|^2_{\alpha}\right)\in[0,\,2\log 2]
$$
and
\begin{align*}
\frac{1}{4K}<\varsigma'<\frac{1}{2K},\quad \varsigma''=2(\varsigma')^2>0.
\end{align*}
We assume that $H$ attains its maximum at the interior point $\mathbf{x}_0\in M.$ It suffices to show that there holds $\lambda_1\leq CK$ at $\mathbf{x}_0$ for some uniform constant $C.$ In what follows we may assume that $\lambda_1\gg K$ at the point $\mathbf{x}_0$ without loss of generality and hence \eqref{ftau} holds.
In the followup, we will calculate at the point $\mathbf{x}_0$ under the local coordinate $(z_1,\cdots,z_n)$ for which $\alpha$ is the identity and $(A_i{}^j)$ is diagonal with entries $A_i{}^i=g_{i\bar i}=\lambda_i$ for $1\leq i\leq n,$ unless otherwise indicated. Note that $(F^{i\bar j})$ is also diagonal at the point $\mathbf{x}_0$ (see \cite{fenggezhengcmabd1}).

Since $\lambda_1$ may not be smooth at $\mathbf{x}_0,$ we introduce a smooth function $\phi$ on $M$ by (cf. \cite[Lemma 5]{brendleetx2017} and \cite[Proof of Theorem 3.1]{twarxiv1906})
 \begin{equation}
 \label{defnphi}
H(\mathbf{x}_0)\equiv \log \phi(\mathbf{x})+\varsigma(|\partial \varphi|_\alpha^2(\mathbf{x})) +\psi(\varphi(\mathbf{x})),\quad \forall\,\mathbf{x}\in M.
\end{equation}
Note that $\phi$ satisfies
\begin{equation}
\phi(\mathbf{x})\geq \lambda_1(\mathbf{x})\quad \forall\,\mathbf{x}\in M,\quad \phi(\mathbf{x}_0)=\lambda_1(\mathbf{x}_0).
\end{equation}
We define a strictly elliptic operator  $L$ which is the same as \eqref{pu} by
\begin{align}
\label{defnl}
L(u)=&F^{ij}(A)\alpha^{\overline{q}j}\left(\partial_i\partial_{\overline{q}}u+W_{i\bar q}(\mathrm{d} u)\right)\\
=&F^{ij}(A)\alpha^{\overline{q}j}\left(\nabla_i\nabla_{\overline{q}}u
+W_{i\bar q}^p(\nabla_p u)+\overline{W_{q\bar i}^p}(\nabla_{\bar p} u)\right),\quad \forall\;u\in C^2(M,\mathbb{R}).\nonumber
\end{align}
Applying the operator $L$ defined in \eqref{defnl} to \eqref{defnphi}, one infers
\begin{align}
\label{0lhatx0}
0=&\frac{1}{\lambda_1}L(\phi)-\frac{1}{\lambda_1^2}F^{i\bar i}|\nabla_i\phi|^2+\varsigma'L(|\partial \varphi|_\alpha^2)+\psi'L(\varphi)+\psi''F^{i\bar i}|\nabla_i\varphi|^2\\
&+\varsigma''F^{i\bar i}
\left|\sum_p\big((\nabla_i\nabla_p\varphi)(\nabla_{\bar p}\varphi)
+(\nabla_p\varphi)(\nabla_{i}\nabla_{\bar p}\varphi)\big)\right|^2
 .\nonumber
\end{align}
Differentiating \eqref{defnphi} one can deduce
\begin{equation}
\label{nalbaiphi}
0=\frac{\nabla_i\phi}{\phi}+\varsigma'\left((\nabla_p\varphi)(\nabla_{\bar p}\nabla_i\varphi +(\nabla_{\bar p}\varphi)\nabla_i\nabla_{p}\varphi\right)+\psi'(\nabla_i\varphi).
\end{equation}
We have (see \cite[Lemma 5.3]{fenggezhengcmabd1})
\begin{lem}
\label{lemddbarphi}
Let $\mu$ denote the multiplicity of the largest eigenvalue of $(A_i{}^j)$ at $\mathbf{x}_0,$ so that
$\lambda_1=\cdots=\lambda_\mu>\lambda_{\mu+1}\geq \cdots\geq \lambda_n.$ Then at $\mathbf{x}_0,$ for each $i$ with $1\leq i\leq n,$ there hold
\begin{align}
\label{nablaiphi}
&\nabla_ig_{k\bar\ell}= (\nabla_i\phi) \alpha_{k\bar \ell},\quad \mbox{for}\quad 1\leq k,\,\ell\leq \mu,\\
\label{ddbarphi}
&\nabla_{\bar i}\nabla_i\phi \geq \nabla_{\bar i}\nabla_{i}g_{1\bar 1}+\sum_{q>\mu}\frac{\left|\nabla_ig_{q\bar 1}\right|^2+\left|\nabla_{\bar i}g_{q\bar 1}\right|^2}{\lambda_1-\lambda_q}.
\end{align}
\end{lem}
A direct calculation, together with the Ricci identity and the first Bianchi identity, yields that (see for example \cite{twcrelle})
\begin{align}
\label{ricciid2times}
\nabla_{\bar \ell}\nabla_k\nabla_{\bar j}\nabla_iu
=&\nabla_{\bar j}\nabla_i\nabla_{\bar \ell}\nabla_ku
+R_{k\bar\ell i}{}^p\nabla_{\bar j}\nabla_pu
-R_{i\bar j k}{}^p\nabla_{\bar\ell}\nabla_pu\\
&-T_{ki}^p\nabla_{\bar\ell}\nabla_{\bar j}\nabla_p u
-\overline{T_{\ell j}^q}\nabla_{k}\nabla_{\bar q}\nabla_i u
-T_{ki}^p\overline{T_{\ell j}^q}\nabla_{\bar q}\nabla_p u.\nonumber
\end{align}
It follows from  \eqref{defnvartheta} and \eqref{ricciid2times} that
\begin{align}
\label{fibaribariithetau1bar1}
F^{i\bar i}\nabla_{\bar i}\nabla_{i}g_{1\bar 1}
=&F^{i\bar i}\nabla_{\bar i}\nabla_{i}\chi_{1\bar 1}
+F^{i\bar i}\nabla_{\bar i}\nabla_i\left(W_{1\bar1}^p\nabla_p\varphi+\overline{W_{1\bar 1}^q}\nabla_{\bar q}\varphi\right)\\
&+F^{i\bar i}\nabla_{\bar 1}\nabla_{1}\nabla_{\bar i}\nabla_{i}\varphi
-F^{i\bar i}\left(T_{i1}^p\nabla_{\bar i}\nabla_{\bar 1}\nabla_p\varphi
+\overline{T_{i1}^q}\nabla_i\nabla_{\bar q}\nabla_{ 1}\varphi\right)\nonumber\\
&+F^{i\bar i}\left(R_{i\bar i1}{}^p\nabla_{\bar 1}\nabla_p\varphi-R_{1\bar 1i}{}^p\nabla_{\bar i}\nabla_p\varphi
-T_{i1}^p\overline{T_{i1}^q}\nabla_{\bar q}\nabla_p\varphi
\right).\nonumber
\end{align}
Differentiating both sides of \eqref{nablakh} by $\nabla_{\bar \ell}$ gives
\begin{equation}
\label{nablabarknablakh}
\nabla_{\bar \ell}\nabla_{k}h
= F^{i\bar j,p\bar q}\left(\nabla_kg_{i\bar j}\right)
\left(\nabla_{\bar\ell}g_{p\bar q}\right)
+F^{i\bar j}
\big(
\nabla_{\bar \ell}\nabla_k\chi_{i\bar j}
+\nabla_{\bar \ell}\nabla_k\nabla_{\bar j}\nabla_i\varphi
+\nabla_{\bar \ell}\nabla_k(W_{i\bar j}(\mathrm{d}\varphi))
\big).
\end{equation}
Substituting \eqref{nablabarknablakh} with $k=\ell=1$ into \eqref{fibaribariithetau1bar1} yields
\begin{align}
\label{fibarinalbabarinablarivarthetau1bar1}
&F^{i\bar i}\nabla_{\bar i}\nabla_{i}g_{1\bar 1}\\
=&-F^{i\bar j,p\bar q}\left(\nabla_1g_{i\bar j}\right)
\left(\nabla_{\bar1}g_{p\bar q}\right)\nonumber\\
 &+F^{i\bar i}\left(\nabla_{\bar i}\nabla_{i}\chi_{1\bar 1}
-\nabla_{\bar 1}\nabla_{1}\chi_{i\bar i}\right)
+F^{i\bar i}\left(\nabla_{\bar i}\nabla_i\left(W_{1\bar1}(\mathrm{d}\varphi)\right)-\nabla_{\bar 1}\nabla_1\left(W_{i\bar i}(\mathrm{d}\varphi)\right)\right)\nonumber\\
&+\nabla_{\bar 1}\nabla_1h-F^{i\bar i}\left(T_{i1}^p\nabla_{\bar i}\nabla_{\bar 1}\nabla_p\varphi
+\overline{T_{i1}^q}\nabla_{i}\nabla_{\bar q}\nabla_1\varphi\right)\nonumber\\
&+F^{i\bar i}\left(R_{i\bar i1}{}^p\nabla_{\bar 1}\nabla_p\varphi-R_{1\bar 1i}{}^p\nabla_{\bar i}\nabla_p\varphi
-T_{i1}^p\overline{T_{i1}^q}\nabla_{\bar q}\nabla_p\varphi
\right).\nonumber
\end{align}
It follows from  \eqref{defnvartheta}, \eqref{nablaiphi} and Young's inequality that
\begin{align}
\label{fibariti1pnablainablabar1nablapurealpart}
&F^{i\bar i}\left(T_{i1}^p\nabla_{\bar i}\nabla_{\bar 1}\nabla_p\varphi
+\overline{T_{i1}^q}\nabla_{i}\nabla_{\bar q}\nabla_1\varphi\right)\\
=&2\Re\left(F^{i\bar i}\overline{T_{i1}^q}\nabla_{i}g_{1\bar q}
-F^{i\bar i}\overline{T_{i1}^q}\left(
\left(\nabla_iW_{1\bar q}^p\right)\nabla_p\varphi
+W_{1\bar q}^p\nabla_i\nabla_p\varphi
+\overline{\nabla_{\bar i}W_{1\bar q}^p}\nabla_{\bar p}\varphi
+\overline{W_{1\bar q}^p}\nabla_i\nabla_{\bar p}\varphi
\right)\right) \nonumber\\
\geq &2\Re\left(F^{i\bar i}\overline{T_{i1}^1}\nabla_{i}g_{1\bar 1}\right)
+2\Re\left(F^{i\bar i}\sum_{q>\mu}\overline{T_{i1}^q}\nabla_{i}g_{1\bar q}\right)\nonumber\\
&-C\lambda_1\mathcal{F}-C\sum_{p}F^{i\bar i}|\nabla_i\nabla_p\varphi|\nonumber\\
\geq&-CF^{i\bar i}|\nabla_{i}g_{1\bar 1}|-\sum_{q>\mu}F^{i\bar i}\frac{|\nabla_{i}g_{1\bar q}|^2}{\lambda_1-\lambda_q}-C\lambda_1\mathcal{F}-C\sum_{p}F^{i\bar i}|\nabla_i\nabla_pu|,\nonumber
\end{align}
where we use the fact that $\lambda_1\gg K>1$ and that both $|\varphi_{i\bar j}|$ and $\lambda_q$ ($q>\mu$) can be controlled by $\lambda_1.$
It follows from   \eqref{eq:c1}, \eqref{lambdamudefn}, \eqref{fk}, \eqref{eq:fk1}, \eqref{eq:fk2} that (see the argument in \cite{stw1503})
\begin{align}
\label{eq:fk3}
&F^{i\bar i}W_{i\bar i}(\nabla g_{1\bar 1})\\
=&\tilde F^{i\bar i}Z_{i\bar i}(\nabla g_{1\bar 1})\nonumber\\
=&\tilde F^{1\bar 1}\sum_{p>1}\left(Z_{1\bar1}^p \nabla_pg_{1\bar 1}+\overline{Z_{1\bar1}^p}\nabla_{\bar p} g_{1\bar 1}\right)
+\sum_{i>1}\tilde F^{i\bar i}Z_{i\bar i}(\nabla g_{1\bar 1})\nonumber\\
\leq&C\sum_{p>1}F^{p\bar p}|\nabla_p g_{1\bar 1}|
+C F^{1\bar 1} \sum_{q }|\nabla_q g_{1\bar 1}|\nonumber\\
=&C\sum_{p>1}F^{p\bar p}|\nabla_p g_{1\bar 1}|
+C F^{1\bar 1}  |\nabla_1 g_{1\bar 1}|
+C F^{1\bar 1} \sum_{q>1 }|\nabla_q g_{1\bar 1}|\nonumber\\
\leq&C\sum_{p>1}F^{p\bar p}|\nabla_p g_{1\bar 1}|
+C F^{1\bar 1}  |\nabla_1 g_{1\bar 1}|
+C \sum_{q>1 }F^{q\bar q}|\nabla_q g_{1\bar 1}|\nonumber\\
\leq&C\sum_pF^{p\bar p}|\nabla_p g_{1\bar 1}|,\nonumber
\end{align}
\begin{equation}
\label{stwacta1fqqbar11wqq}
F^{i\bar i}\nabla_{\bar i}\nabla_i(W_{1\bar 1}(\mathrm{d}\varphi))
\geq
-C\left(F^{i\bar i}\sum_p|\nabla_i\nabla_p\varphi|+\lambda_1\mathcal{F}\right)
\end{equation}
and
\begin{equation}
\label{stwacta1fqqqqwbar11}
F^{i\bar i}\nabla_{\bar 1}\nabla_1(W_{i\bar i}(\mathrm{d}\varphi))
\leq
 C\left(F^{i\bar i}\left(|\nabla_ig_{1\bar 1}|+\sum_p|\nabla_i\nabla_p\varphi|\right)+\lambda_1\mathcal{F}\right),
\end{equation}
where we also use the fact that $\lambda_1\gg K>1$ and that $|\varphi_{i\bar j}|$ can be controlled by $\lambda_1.$
Note that \eqref{stwacta1fqqbar11wqq} and \eqref{stwacta1fqqqqwbar11} can be found in \cite{stw1503} directly (cf. \cite{zhengimrn}).

Applying the operator $L$ defined in \eqref{defnl} to $\phi,$ we can deduce from \eqref{ftau}, \eqref{nablaiphi}, \eqref{ddbarphi}, \eqref{fibarinalbabarinablarivarthetau1bar1}, \eqref{fibariti1pnablainablabar1nablapurealpart}, \eqref{eq:fk3}, \eqref{stwacta1fqqbar11wqq} and \eqref{stwacta1fqqqqwbar11} that
\begin{align}
L(\phi)
\label{estimateoflphi}
=&F^{i\bar i}\left(\nabla_{\bar i}\nabla_i\phi+W_{i\bar i}(\nabla\phi)\right)\\
\geq&F^{i\bar i}\nabla_{\bar i}\nabla_ig_{1\bar1}
+F^{i\bar i}W_{i\bar i}(\nabla g_{1\bar1})\nonumber\\
&+\sum_{q>\mu}F^{i\bar i}\frac{|\nabla_ig_{q\bar 1}|^2+|\nabla_{\bar i}g_{q\bar 1}|^2}{\lambda_1-\lambda_q}\nonumber\\
\geq&-F^{i\bar j,p\bar q}\left(\nabla_1g_{i\bar j}\right)
\left(\nabla_{\bar1}g_{p\bar q}\right)\nonumber\\
&+2\Re\left(F^{i\bar i}\overline{T_{i1}^q}\nabla_{i}g_{1\bar q}\right)
+\sum_{q>\mu}F^{i\bar i}\frac{|\nabla_ig_{q\bar 1}|^2+|\nabla_{\bar i}g_{q\bar 1}|^2}{\lambda_1-\lambda_q}\nonumber\\
&-C F^{i\bar i}\left(|\nabla_ig_{1\bar 1}|+\sum_p|\nabla_i\nabla_p\varphi|\right)-C\lambda_1\mathcal{F}  \nonumber\\
\geq&-F^{i\bar j,p\bar q}\left(\nabla_1g_{i\bar j}\right)
\left(\nabla_{\bar1}g_{p\bar q}\right)\nonumber\\
&-C F^{i\bar i}\left(|\nabla_ig_{1\bar 1}|+\sum_p|\nabla_i\nabla_p\varphi|\right)-C\lambda_1\mathcal{F}.  \nonumber
\end{align}
Here we remind that $\nabla_1\nabla_{\bar 1}h =O(\lambda_1)$ is absorbed into $C\lambda_1\mathcal{F}$ by \eqref{ftau}.

From \eqref{0lhatx0} and \eqref{estimateoflphi}, one can infer that
\begin{align}
\label{0lhatx01}
0\geq&-\frac{1}{\lambda_1}F^{i\bar j,p\bar q}\left(\nabla_1g_{i\bar j}\right)
\left(\nabla_{\bar1}g_{p\bar q}\right)
-\frac{1}{\lambda_1^2}F^{i\bar i}|\nabla_i\phi|^2  \\
&+\varsigma'L(|\partial \varphi|_g^2)+\varsigma''F^{i\bar i}
\left|\sum_p\big((\nabla_i\nabla_p\varphi)(\partial_{\bar p}\varphi)
+(\nabla_p\varphi)(\nabla_{i}\nabla_{\bar p}\varphi)\big)\right|^2
  \nonumber\\
&+\psi'L(\varphi)+\psi''F^{i\bar i}|\nabla_i\varphi|^2-\frac{C}{\lambda_1} F^{i\bar i}\left(|\nabla_ig_{1\bar 1}|+\sum_p|\nabla_i\nabla_p\varphi|\right)-C \mathcal{F}.\nonumber
\end{align}
On the other hand, we can still have (see \cite[Lemma 5.2]{fenggezhengcmabd1})
\begin{lem}
\label{lemlpartialug2}
There exists a uniform constant $C>0$ such that
\begin{align}
\label{lpartialug2}
 L(|\partial \varphi|_\alpha^2)
=&\sum_kF^{i\bar i}\left(|\nabla_i\nabla_k\varphi|^2+|\nabla_i\nabla_{\bar k}\varphi|^2 \right)+2\Re\left(\sum_k\left(\nabla_k\varphi\right)\left(\nabla_{\bar k}h\right)\right) \\
&+F^{i\bar i}\left(\nabla_k\varphi\right)\overline{T_{ki}^p}(\nabla_i\nabla_{\bar p}\varphi)
+F^{i\bar i}\left(\nabla_{\bar k}\varphi\right)T_{ki}^p(\nabla_p\nabla_{\bar i}\varphi)
+O(|\partial u|_g^2)\mathcal{F}\nonumber\\
\geq&\sum_kF^{i\bar i}\left(|\nabla_i\nabla_k\varphi|^2+(1-\varepsilon)|\nabla_i\nabla_{\bar k}\varphi|^2 \right)\nonumber\\
&+2\Re\left(\sum_k\left(\nabla_k\varphi\right)\left(\nabla_{\bar k}h\right)\right)
-C\varepsilon^{-1}|\partial \varphi|_\alpha^2 \mathcal{F},\nonumber
\end{align}
where $\varepsilon$ is an arbitrary constant with $ \varepsilon\in (0,1/2].$
\end{lem}
Note that the term
$$
2\Re\left(\sum_k\left(\nabla_k\varphi\right)\left(\nabla_{\bar k}h\right)\right)
$$
which is $O(K)$
can be absorbed into $C\varepsilon^{-1}|\partial \varphi|_\alpha^2 \mathcal{F}$ with a larger uniform constant $C$ by \eqref{ftau} since we assume that $K\gg1$.
Since $\varsigma'\geq 1/(4K),$  it follows from \eqref{lpartialug2} with $\varepsilon=1/3$ and Young's inequality that
\begin{align}
\label{lgradientusquareuse2}
&\varsigma'L(|\partial \varphi|_\alpha^2)-\frac{C}{\lambda_1} F^{i\bar i}\sum_p|\nabla_i\nabla_p\varphi|\\
\geq&\frac{1}{6K}\sum_kF^{i\bar i}\left(|\nabla_i\nabla_k\varphi|^2 +|\nabla_i\nabla_{\bar k}\varphi|^2 \right)-C\mathcal{F},\nonumber
\end{align}
where we also use the fact that $\lambda_1\gg K>1.$
Hence we can deduce from \eqref{0lhatx01} and \eqref{lgradientusquareuse2} that
\begin{align}
\label{0lhatx02}
0\geq&-\frac{1}{\lambda_1}F^{i\bar j,p\bar q}\left(\nabla_1g_{i\bar j}\right)
\left(\nabla_{\bar1}g_{p\bar q}\right)
-\frac{1}{\lambda_1^2}F^{i\bar i}|\nabla_i\phi|^2  \\
&+\frac{1}{6K}\sum_kF^{i\bar i}\left(|\nabla_i\nabla_k\varphi|^2+|\nabla_i\nabla_{\bar k}\varphi|^2 \right)\nonumber\\
&+\varsigma''F^{i\bar i}
\left|\sum_p\big((\nabla_i\nabla_p\varphi)(\nabla_{\bar p}\varphi)
+(\nabla_p\varphi)(\nabla_{i}\nabla_{\bar p}\varphi)\big)\right|^2
  \nonumber\\
&+\psi'L(\varphi)+\psi''F^{i\bar i}|\nabla_i\varphi|^2-\frac{C}{\lambda_1} F^{i\bar i} |\nabla_ig_{1\bar 1}| -C \mathcal{F}.\nonumber
\end{align}
From \eqref{nablaiphi}, we know that $\nabla_i\phi=\nabla_ig_{1\bar 1}.$ This, together with \cite[Equation (3.7)]{stw1503}, yields that \eqref{0lhatx02} is the same as \cite[Equation (3.28)]{stw1503} essentially after changing  $\nabla_p\nabla_p\varphi$ and $\nabla_ig_{1\bar 1}$ into $\partial_i\partial_p \varphi$ and $\partial_ig_{1\bar 1}$ respectively, and changing the coefficient of $\mathcal{F}$ into a larger uniform constant (only replacing $\tilde\lambda_{1,k}$ in \cite{stw1503} with $\nabla_k\phi $, $\phi$ in \cite{stw1503} by $\varsigma$ and $F^{k\bar k}u_{k\bar k}$ in \cite{stw1503} by $L(\varphi)$).
After changing these notations, we can repeat the argument in \cite{stw1503} word for word to get
\begin{equation}
\lambda_1\leq CK,
\end{equation}
by replacing $\tilde H_k=0$ in \cite{stw1503} by \eqref{nalbaiphi} and replacing the paragraph between \cite[Inequality (3.53)]{stw1503}(not containing) and \cite[Inequality (3.54)]{stw1503}(containing) by $$\psi'L(\varphi)= \psi'F^{k\bar k}(g_{k\bar k}-\chi_{k\bar k}).$$
\textbf{For the first order estimate}, we need prove
\begin{equation}
\label{thmfirstorderwubian}
\sup_M|\partial u|_g\leq C.
\end{equation}
 We use the blowup argument in \cite{gaborjdg,twjams} originated from \cite{dkajm}. Since  $h$ still uniformly bounded, the argument in \cite{gaborjdg,twcrelle,fenggezhengcmabd1} still works without any modification. We omit the details here.

Given \eqref{equzeroestimate}, \eqref{equ2ndestimate} and \eqref{thmfirstorderwubian}, $C^{2,\gamma}$-estimate for some $0<\gamma<1$ follows from the Evans-Krylov theory \cite{evanscpam1982,krylov1982,trudingertransactionams1983} (see also \cite{twwycvpde}). Differentiating the equations and using the Schauder theory (see for example \cite{gt1998}), we then deduce uniform a priori $C^k$ estimates for all $k\geq 0.$

Uniqueness follows from the maximum principle as in the arguments in Corollary \ref{corunique}.

\end{proof}

\end{document}